\documentclass[12pt]{article}
\usepackage{amssymb,amsfonts,amsmath,amsthm, breqn, color, float, mathtools, booktabs, hvfloat, url, mathrsfs}
\usepackage{adjustbox}
\usepackage{caption}
\newcommand{\witi}{\widetilde}

\newcommand{\cc}{{\mathbb C}}

\newcommand{\calg}{{\mathcal G}}

\newcommand{\calr}{{\mathcal R}}

\newcommand{\calb}{{\mathscr B}}

\newcommand{\calo}{{\mathcal O}}

\newcommand{\beq}{\begin{eqnarray*}}
    \newcommand{\feq}{\end{eqnarray*}}
\newcommand{\beqn}{\begin{eqnarray}}
\newcommand{\feqn}{\end{eqnarray}}

\newtheorem{theorem}{Theorem}
\newtheorem*{conj*}{Conjecture}
\makeatletter \@addtoreset{theorem}{section}\makeatother

\newcommand{\nn}{{\mathbb N}}
\makeatletter \@addtoreset{theorem}{section}\makeatother

\newcommand{\calv}{{\mathcal V}}
\makeatletter \@addtoreset{theorem}{section}\makeatother

\newtheorem{lemma}[theorem]{Lemma}

\newtheorem*{theorem*}{Theorem}

\newtheorem{proposition}[theorem]{Proposition}

\newtheorem{example}[theorem]{Example}
\usepackage{algorithm}
\usepackage{graphicx}
\usepackage[noend]{algpseudocode}
\makeatletter
\def\BState{\State\hskip-\ALG@thistlm}
\makeatother

\DeclareCaptionLabelFormat{noname}{#2}
\newlength\myindent
\title{Horizontal visibility graph of a random restricted growth sequence}

\author{Toufik~Mansour\thanks{ Department of Mathematics, University of Haifa, 199 Abba Khoushy Ave, 3498838 Haifa, Israel;
        \newline e-mail: tmansour@univ.haifa.ac.il} \and
        Reza~Rastegar\thanks{Occidental Petroleum Corporation, Houston, TX 77046 and Departments of Mathematics and Engineering, University of Tulsa, OK 74104, USA - Adjunct Professor; e-mail:  reza\_rastegar2@oxy.com}
        \and
        Alexander~Roitershtein \thanks{Department of Statistics, Texas A\&M University, College Station, TX 77843, USA; \newline e-mail: alexander@stat.tamu.edu}
}

\begin{document}
\maketitle
\begin{abstract}
We study the distributional properties of horizontal visibility graphs associated with random restrictive growth sequences and random set partitions of size $n.$ Our main results are formulas expressing the expected degree of graph nodes in terms of simple explicit functions
of a finite collection of Stirling and Bernoulli numbers.
\end{abstract}


\section{Introduction and statement of results}
We study here horizontal visibility graphs of restricted growth sequences. The latter class of sequences is of interest both independently and in connection with set partitions \cite{Mb}, $q$-analogues \cite{maybe1}, certain combinatorial matrices \cite{galvin},  bargraphs \cite{bargraph}, and Gray codes \cite{conflitti1}.
\par
A horizontal visibility graph (HVG)  \cite{luque3} constitutes a paradigmatic complex network representation of sequential data,  typically used to reveal order structures within the data set \cite{complex, survey}. HVG-based algorithms have been employed to characterize fractal behavior of dynamical systems \cite{frac, bhurst}, study canonical routes to chaos (see \cite{chaos} and references therein), discriminate between chaotic and stochastic time series \cite{dnc}, and test time series irreversibility \cite{erev}. There is a growing body of literature using these combinatorial data analysis techniques in applied fields such as optics \cite{laser}, fluid dynamics \cite{turba},  geophysics \cite{seim}, physiology and neuroscience \cite{madl, neuro}, finance \cite{finance1}, image processing \cite{image}, and more \cite{complex, survey}. For other graph theoretic methods of applied time series analysis as well as many fruitful extensions of the horizontal visibility algorithm, we refer to recent surveys \cite{complex, survey}.
\par
From a combinatoric point of view, HVGs are outerplanar graphs with a Hamiltonian path, an important subclass of so-called \textit{non-crossing graphs} of algebraic combinatorics \cite{gutin}. An illuminating characterization of HVGs  using ``one-point compactified" times series and tools of algebraic topology   is obtained in a recent work \cite{cute}. Theoretical body of work on the HVGs includes studies of their degree distributions \cite{degree1, degree}, information-theoretic \cite{info, math1} and other \cite{peak} topological characteristics, motifs \cite{motif, math3}, spectral properties \cite{math, luque3}, and dependence of graph features on the parameter for a specific parametric family of chaotic \cite{wuba} or stochastic processes \cite{bhurst, stable}. For more, see a recent comprehensive survey \cite{survey} and an extensive review of earlier results \cite{nunez}.
\par
In this paper, our main focus is on the degree properties of the  horizontal visibility graph associated with a random restricted growth sequence.
Let $\pi= \pi_1\cdots \pi_n$ be a sequence of elements of a totally ordered set.  We say that $(\pi_i, \pi_j)$ is a \textit{strong visible pair} if $$\max_{i<\ell<j}\pi_\ell<\min\{\pi_i, \pi_j\},$$ where we use the usual convention that $\max \emptyset = -\infty.$ Similarly, we refer to $(\pi_i, \pi_j)$ as a \textit{weak visible pair} if $$\max_{i<\ell<j}\pi_\ell\leq\min\{\pi_i, \pi_j\}.$$ We denote by $\calv_\pi$ the set of all strong visible pairs of $\pi,$ and let $V_\pi = \mbox{Card}(\calv_\pi)$ be the number of strong visible pairs in the sequence $\pi$. For example, $$\calv_{12122}=\{(1,2), (2,3), (3,4), (4,5), (2,4)\}, \quad V_{12122}=5.$$ We use the above notation with addition of the superscript $w$ to denote the corresponding weak visibility pairs statistics. For example, $$\calv^{w}_{12122}=\{(1,2), (2,3), (3,4), (4,5), (2,4), (2,5)\}, \quad V^{w}_{12122}=6.$$ The graph $\calg_\pi:=([n], \calv_\pi)$ with the set of nodes $[n]:=\{1,2,\ldots,n\}$ is the \textit{horizontal visibility graph} associated with $\pi$ \cite{luque3}. For $i \in [n],$ we denote by $d_\pi(i)$ the degree of the node $i$ in the visibility graph $\calg_\pi.$ We set $e_\pi (i,j)=1$ when $(i,j)\in \calv_\pi$ and $e_\pi (i,j)=0$ otherwise. Thus,
\beqn
\label{d_n_exp}
d_\pi(i) = \sum_{j\in [n]\setminus \{i\}} e_\pi(i,j).
\feqn
We now turn to the definition of a restricted growth sequence. A sequence of positive integers $\pi=\pi_1\pi_2\cdots \pi_n\in \nn^n$ is called
a restricted growth sequence if
\beq
\pi_1 = 1\qquad \mbox{\rm and}  \qquad \pi_{j+1} \leq 1 + \max\{\pi_1,\cdots,\pi_j\} \ \ \ \text{for all} \ 1\leq j <n.
\feq
There is a bijective connection between these sequences and canonical set partitions.
A \textit{partition} of a set $A$ is a collection of non-empty, mutually disjoint subsets, called \textit{blocks},
whose union is the set $A$. A partition $\Pi$ with $k$ blocks is called a \textit{$k$-partition} and denoted by $\Pi = A_1|A_2|\cdots|A_k$.
A $k$-partition $A_1|A_2|\cdots |A_k$ is said to be in the \textit{standard form} if the blocks $A_i$ are labeled in such a way that
\beq
\min A_1 < \min A_2<\cdots< \min A_k.
\feq
The partition can be represented equivalently by the \textit{canonical sequential form} $ \pi_1\pi_2\ldots\pi_n,$ where
$\pi_i\in [n]$ and $i\in A_{\pi_i}$ for all $i$ \cite{Mb}. In words, $\pi_i$ is the label of the partition block that contains $i.$
It is easy to verify that a word $\pi\in [k]^n$ is a canonical representation of a $k$-partition of $[n]$ in the standard form
if and only if it is a restricted growth sequence \cite{Mb}.
\begin{example}
For instance, canonical partition $\{1,4,7\}\,|\, \{2,3,6,9\}\,|\, \{5,8\}$ in the canonical sequential form is $\pi=122132132,$ where $\pi_3=2$ indicates that $3$ belongs to the second block $\{2,3,6,9\},$ etc. The (weak and strong) visibility graphs of $\pi$ are given in Fig.~\ref{figvisgraphs1} below.
\begin{figure}[ht!]
	\begin{center}
		\begin{picture}(280,70)(30,-60)
		\setlength{\unitlength}{1.mm}
		\put(0,-5){\multiput(0,0.04)(5,0.04){11}{\put(2.5,-10){$\bullet$}}
			\put(2.5,-8.5){\line(1,0){50}}
			\qbezier(9,-8)(13,-1)(18,-8)\qbezier(33,-8)(41,1)(48,-8)
			\qbezier(19,-8)(25,1)(33,-8)\qbezier(24,-8)(28,-3)(33,-8)}
		\put(70,0){
			\put(0,-5){\multiput(0,0.04)(5,0.04){11}{\put(2.5,-10){$\bullet$}}
				\put(2.5,-8.5){\line(1,0){50}}
				\qbezier(9,-8)(13,-1)(18,-8)\qbezier(33,-8)(41,1)(48,-8)
				\qbezier(19,-8)(25,1)(33,-8)\qbezier(24,-8)(28,-3)(33,-8)
				\qbezier(19,-8)(34,11)(48,-8)}
		}
		\end{picture}
	\end{center}
\caption{On he left is a picture of the strong visibility graph of the sequence $12132132231.$ On the right, is the weak visibility graph associated with same sequence. }\label{figvisgraphs1}
\end{figure}
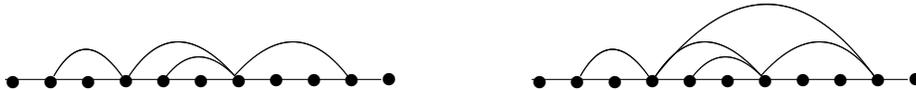
\end{example}
We denote by $\calr_n$ the set of all restricted growth sequences of length $n.$ For a given $\pi\in \calr_n$, we let $\calo(\pi):=\mbox{Card}\{\pi_i: i\in [n]\},$ the number of different letters in the word $\pi.$ For example, $\calo(1231) = 3$.  We denote by $\calr_{n,k}$ the set of all restricted growth sequences $\pi$ with $\calo(\pi)=k.$ Clearly, $\calr_n:=\bigcup_{k\in [n]}\calr_{n,k}.$
\par
It is well-known that $\mbox{Card}(\calr_{n,k})= S_{n,k}$ and $\mbox{Card}(\calr_n)= B_n$ where $S_{n,k}$ is a Stirling number of second kind and $B_n$ is the $n$-th Bell number \cite{Mb}. The Stirling numbers can be introduced algebraically in several different ways. For instance,
\beqn
\label{SnkIdent}
\frac{x^k}{\prod_{j=1}^k(1-jx)}=\sum_{n\geq0}S_{n,k}x^n,\qquad \qquad \forall\,k\in\nn.
\feqn
Alternatively, one can define the sequence of Stirling numbers of the second kind as the solution to the recursion
\beqn
\label{compa}
S_{n,k}=kS_{n-1,k}+S_{n-1,k-1},\qquad n,k\in\nn,\,k\leq n,
\feqn
with $S_{0,0}=1$ and $S_{0,n}=0.$ The sequence of Bell numbers $(B_n)_{n\geq 0}$ can be then defined, for instance, through the formula
$B_n=\sum_{k=0}^n S_{n,k},$ or, recursively via the formula $B_{n+1}=\sum_{k=0}^n \binom{n}{k}B_k$
with $B_0=1,$ or through Dobinski's formula \cite{comtet}
\beqn
\label{doob}
B_n=\frac{1}{e}\sum_{m=0}^{\infty}\frac{m^n}{m!}, \qquad\qquad n\geq 0.
\feqn
In what follows, we denote a random restricted growth sequence, sampled uniformly from $\calr_{n,k}$ (resp. $\calr_{n}$) by $\pi^{(n)}$ (resp. $\pi^{(n,k)}$). That is,
\beq
P(\pi^{(n,k)} = \pi) = \frac{1}{S_{n,k}} \quad \mbox{for all} \quad \pi \in \calr_{n,k},
\feq
and
\beq
P(\pi^{(n)}  = \pi) = \frac{1}{B_n} \quad \mbox{for all} \quad \pi \in \calr_{n}.
\feq
\par
We denote by $\calg_{n,k}:=\calg_{\pi^{(n,k)}}$ (resp. $\calg_n:=\calg_{\pi^{(n)}}$) the HVG of the random restrictive
growth sequence $\pi^{(n,k)}$ (resp. $\pi^{(n)}$). Furthermore, we use the notations $e_n(.,.),$ $d_n(.),$ and $V_n$ to denote, respectively, $e_{\pi^{(n)}}(.,.),$ $d_{\pi^{(n)}}(.),$ and $V_{\pi^{(n)}}.$ See Fig.~\ref{fig:gn_instances} below for two instances of visibility graphs of uniformly sampled restrictive growth sequences of length $n=200.$

\begin{figure}
\centering
\begin{tabular}{ccc}
\includegraphics[scale=0.5]{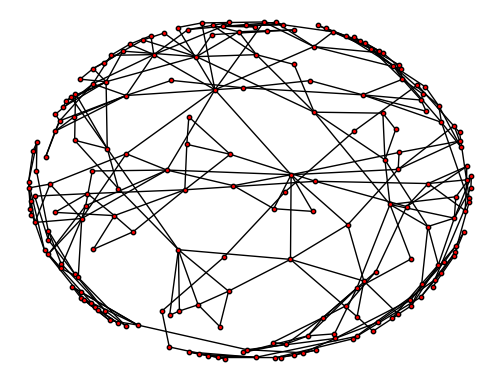}&\qquad$\mbox{}$\qquad&
\includegraphics[scale=0.5]{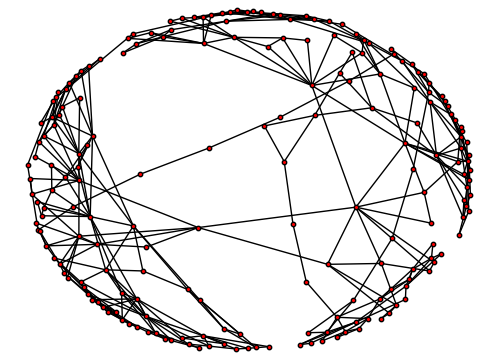}
\end{tabular}
\caption{An instance of $\calg_{200}$ (on the left) and the corresponding $\calg_{200}^w$ (on the right).}
\label{fig:gn_instances}
\end{figure}
\par
For any $k\in\nn,$ we define the generating function
\beqn
P_k(x,q) := \sum_{n=k}^\infty x^n S_{n,k} E\big(q^{V_\pi}|\pi\in\calr_{n,k}\big)
= \sum_{n=k}^\infty  \sum_{\pi \in \calr_{n,k}} x^n q^{V_\pi}, \quad x, q\in \cc. \label{Pkxq-def}
\feqn
Knowing an explicit form of \eqref{Pkxq-def}, would in principle give us the distribution of $V_n$ in full details for all $n\in\nn.$ Unfortunately,
so far we were unable to find an explicit form of $P_k(x,q).$ In this paper, we calculate instead the following generating function:
\beq
\underline Q(x,y):= \sum_{k\geq1} y^k \sum_{n\geq 0} \frac{x^n}{n!} B_n E(V_\pi|\pi\in\calr_{n,k}).
\feq
\begin{theorem}
\label{thm1}
We have:
\beq
\underline Q(x,y)= \frac{1}{y}\int_0^xe^{-ye^{x-t}-t}\int_0^te^{ye^{x-r}+r}(e^{r-x}+y)\underline T(r,ye^{x-r})\,drdt,
\feq
where
\beq
\underline T(x,y)&=&y^3\int_0^x(x-t)e^{ye^t-y}\int_0^tEi(1,ye^r)e^{ye^r+2r}\,drdt
\\
&&
\qquad
+y\int_0^x(t-x)e^{ye^t-y}(Ei(1,ye^t)e^{ye^t}(ye^t-1)-ye^t)\,dt
\\
&&
\qquad\qquad
+y(1-y)\int_0^x(t-x)e^{ye^t}Ei(1,ye^t)\,dt,
\feq
and $Ei(1,z)=\int_1^{\infty}\frac{e^{-zt}}{t}\,dt$ is the exponential integral.
\end{theorem}
\begin{example}
First several terms of the generating function $\underline Q(x,1)$ are given by
\beq
x^2+\frac{5}{3}x^3+\frac{47}{24}x^4+\frac{113}{60}x^5+\frac{19}{12}x^6+\frac{1013}{840}x^7
+\frac{11429}{13440}x^8+\frac{204361}{362880}x^9 + \cdots
\feq
\end{example}
$\mbox{}$
\par
The proof of Theorem~\ref{thm1} is given in Section~\ref{proof1}. The solution is derived from a PDE for $\underline T$ which is obtained
in Lemma~\ref{lem4}. Our next result, Theorem~\ref{thm2}, gives a an alternative, closed form expression for $E(V_n)$ through a different, probabilistic approach.
\par
We partition $I_n= \{(i,j):1\leq i<j\leq n\}$ into three subsets
\beq
I_n^{(1)} &:=& \{ (1,j)\ : \ 3\leq j \leq n \}, \\
I_n^{(2)} &:=& \{ (i,i+1)\ : \ 1\leq i \leq n-1 \}, \\
I_n^{(3)} &:=& \{ (i,j)\ : \ 2 \leq i < j \leq n, j >  i+1 \}.
\feq
Clearly, $e_\pi(i,j)=0$ on $I_n^{(1)}$ and $e_\pi(i,j)=1$ on $I_n^{(2)}$ for all $\pi\in \calr_n.$ Therefore,
\beqn
\label{V_n_exp}
V_\pi = n-1 + \sum_{(i,j)\in I_n^{(3)}} e_{\pi}(i,j).
\feqn
The following theorem evaluates the probability that $(i, j)\in \calv_n$ for a given $(i, j)\in I_n^{(3)}$ in terms of explicit multi-linear polynomials
of $S_{k,i},$ $B_i,$ and Bernoulli numbers. By virtue of \eqref{d_n_exp}, the result immediately yields the average degree $E\big(d_n(i)\big)$ of any given node $i\in V_n$ and the average number of edges $E(V_n).$
\par
We  will use the following Bernoulli formula for Faulhaber polynomials \cite{comtet}:
\beqn
\label{rho}
\Psi_n(t):=\sum_{k=0}^t k^{n-1}=\frac{1}{n}
\sum_{\ell=0}^{n-1} \binom{n}{\ell} t^{n-\ell} \calb_\ell,\qquad \qquad n\in\nn,t\geq 0.
\feqn
where $\calb_\ell$ are Bernoulli numbers. The latter can be calculated, for example, using the recursion
\beq
\sum_{\ell=0}^{n-1}\binom{n}{\ell}\calb_\ell=0
\feq
with $\calb_0=1.$ See, for instance, \cite{comtet} for alternative definitions of Bernoulli numbers.
\par
We will also need the following well-known extension of Dobinski's identity \eqref{doob}.
For any integers $n,t\geq 0$ we have:
\beqn
\nonumber
\Theta_n(t)&:=&\frac{1}{e}\sum_{m=t}^{\infty} \frac{m^n}{(m-t)!} =
\frac{1}{e}\sum_{k=0}^{\infty} \frac{(k+t)^n}{k!} =
\frac{1}{e}\sum_{\ell=0}^n \binom{n}{\ell} t^{n-\ell}\sum_{k=0}^{\infty} \frac{k^\ell}{k!}
\\
&=&
\sum_{\ell=0}^n \binom{n}{\ell} t^{n-\ell} B_\ell, \label{Blextension}
\feqn
were in the last step we applied the original formula \eqref{doob}.
\begin{theorem} \label{thm2}
For all $n\geq 3$ and $(i,j)\in I_n^{(3)},$ we have
\beq
&& B_nP\big((i,j) \in \calv_n\big) = \sum_{t=1}^{i-1} S_{i-1,t}\Theta_{n-j+1}(t)\Psi_{j-i}(t-1)    \\
&& \quad + \sum_{t=1}^{i-1} S_{i-1,t} \Theta_{n-j}(t) \sum_{a=1}^t \big\{ -a(a-1)^{j-i-1}+\Psi_{j-i}(a-1) \big\}   \\
&& \quad +  \sum_{t=1}^{i-1} S_{i-1,t} \Theta_{n-j+1}(t+1)t^{j-i-1}+\sum_{t=1}^{i-1} S_{i-1,t} \Theta_{n-j}(t+1)\big\{-(t+1)t^{j-i-1}+\Psi_{j-i}(t)\big\}.
\feq
\end{theorem}
\par
\begin{figure}[h!b]
\centering
\begin{minipage}{.5\textwidth}
\centering
\includegraphics[scale=.4]{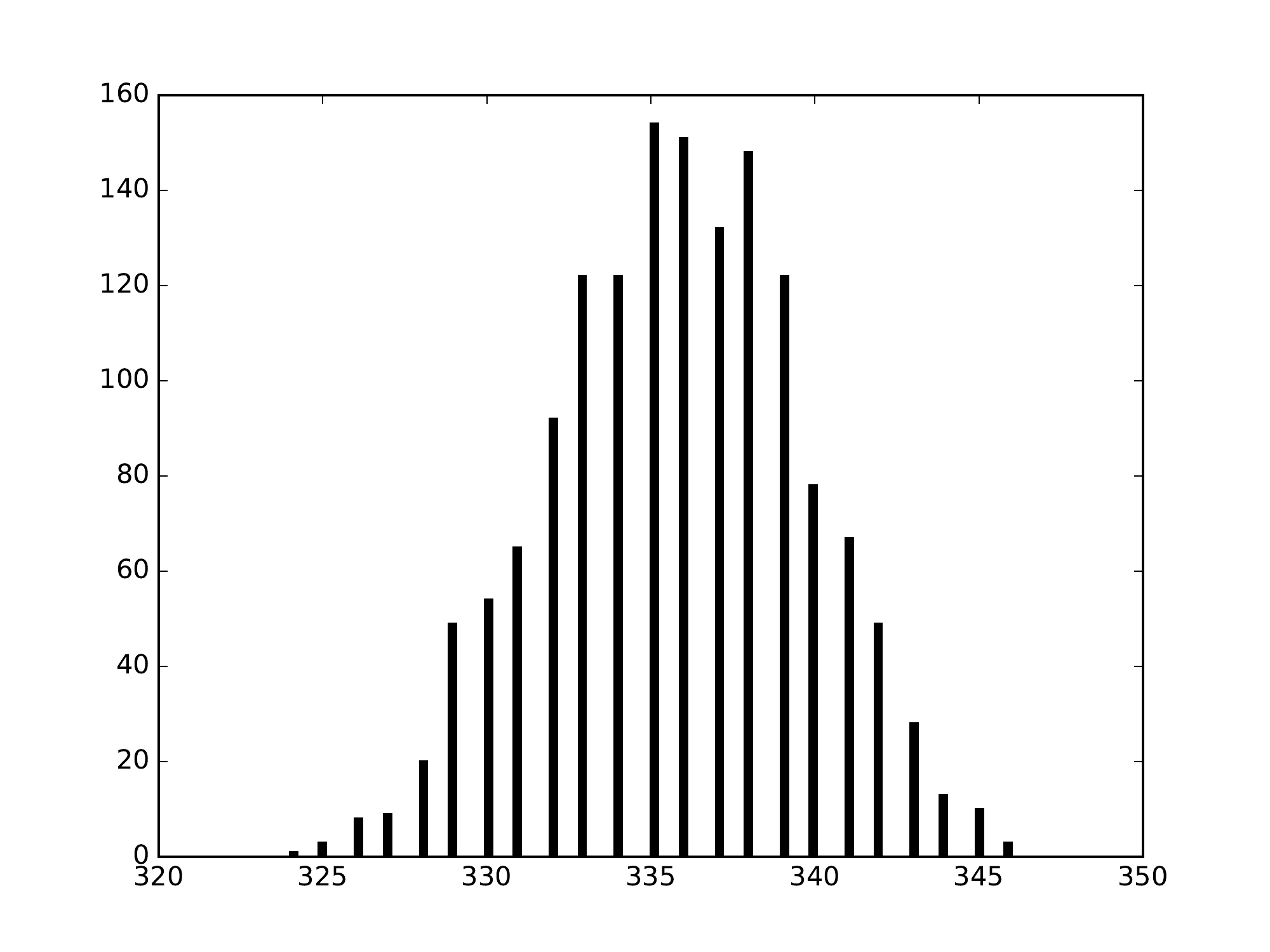}
\end{minipage}%
\begin{minipage}{.5\textwidth}
\centering
\includegraphics[scale=0.4]{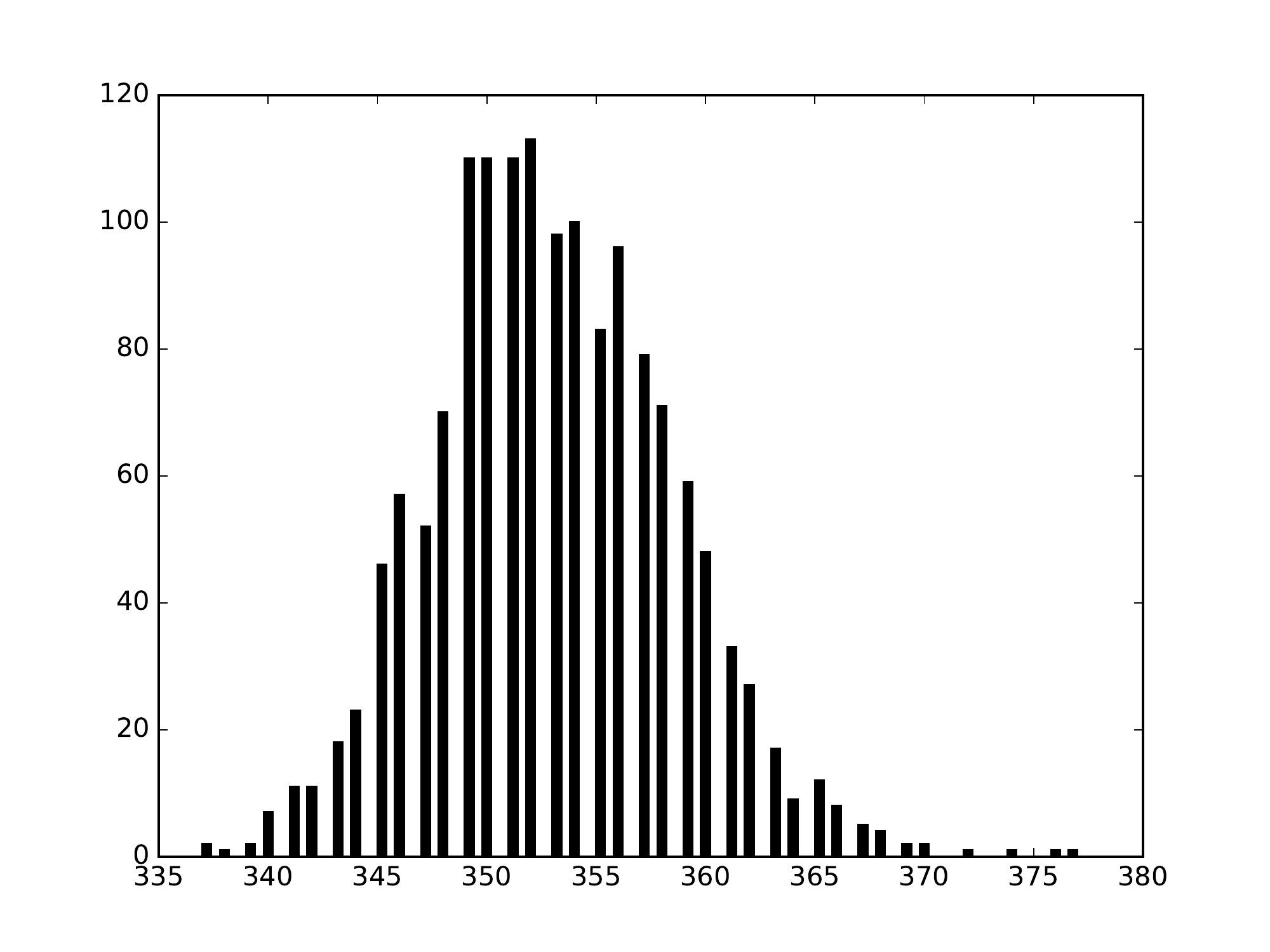}
\end{minipage}
\caption{Empirical distributions of $V_{200}$ (left) and  $V^w_{200}$ (right) based on $1500$ samples. }
\end{figure}
\par
The proof of Theorem~\ref{thm2} is deferred to Section~\ref{pr2}.
We next evaluate the probability that for a given pair of nodes $i,j\in [n],$ we have $(i,j)\in \calv^w_n$ but $(i,j)\notin \calv_n$.
\begin{theorem}
\label{thm3}
The following holds true for $n\geq 3:$
\item [(i)] If $(i,j) \in I_n^{(3)},$ then
\beq
&&B_n P((i,j)\in \calv^w_n\setminus \calv_n) =
\\
&&
\qquad
\sum_{t=1}^{i-1} S_{i-1,t} \Theta_{n-j}(t)\big\{-t^{j-i} + t^{j-i-1}+2\Psi_{j-i}(t-1) \big\}   \\
&& \qquad \quad +\sum_{t=1}^{i-1} S_{i-1,t}\Theta_{j-i+1}(t) t^{j-i-1}
- \sum_{t=1}^{i-1} S_{i-1,t} \Theta_{j-i}(t+1)\big\{(t-1)t^{j-i-1}  +t(t+1)^{j-i-1}\big\} \\
&& \qquad \quad + \sum_{t=1}^{i-1} S_{i-1,t}\Theta_{j-i+1}(t+1)\big\{(t+1)^{j-i-1}-t^{j-i-1}\big\} .
\feq
\item [(ii)] If $(i,j) \in I_n^{(1)},$ then
\beq
P((i,j)\in \calv^w_n\setminus \calv_n) = \frac{B_{n-j+i+1}}{B_n}.
\feq
\end{theorem}
$\mbox{}$
\par
Since $\calv_n \subset \calv_n^w,$ we have
$$E(V_n^w) = E(V_n) + \sum_{(i,j)\in I_n^{(1)}} P((i,j) \in \calv^w_n\setminus \calv_n) + \sum_{(i,j)\in I_n^{(3)}} P((i,j) \in \calv^w_n\setminus \calv_n),$$
which yields $E(V_n^w).$ The proof of Theorem~\ref{thm3} is included in Section~\ref{thm3p}.
\section{Proof of Theorem~\ref{thm1}}
\label{proof1}
Throughout this section, for any given ordinary generating function $A(x)=\sum_{n\geq0} a_n x^n,$ $x\in \cc,$ we use $\underline A$ to denote the corresponding exponential generating function. That is,
\beq
\underline A(x):=\sum_{n=0}^\infty a_n\frac{x^n}{n!} = \sum_{n=0}^\infty\frac{x^n}{n!} [x^n]A(x),
\feq
where $[x^n]A(x)$ stands for the coefficient of $x^n$ in the generating function $A(x).$
\par
Note that each restricted growth sequence in $\calr_{n,k}$ can be represented as a word in the form $1\pi^{(1)}2\pi^{(2)}\cdots k\pi^{(k)},$
where $\pi^{(j)}$ is an arbitrary subword over the alphabet $[j].$ Therefore, we can rewrite \eqref{Pkxq-def} as
\beqn
\label{eqP1}
P_k(x,q)=x^kL_k(x,q)\prod_{j=1}^{k-1}M_j(x,q),
\feqn
where $L_k(x,q)$ and $M_k(x,q)$ are given by
\beqn
\label{lm}
\begin{array}{rl}
L_k(x,q)&=\sum_{n\geq0}\sum_{\pi\in[k]^n}x^nq^{V(k\pi)},
\\
[2mm]
M_k(x,q)&=\sum_{n\geq0}\sum_{\pi\in[k]^n}x^nq^{V(k\pi(k+1))}.
\end{array}
\feqn
This representation is instrumental in our proof of the following result:
\begin{proposition}
\label{mth1}
For $k\geq 1$,
\beq
P_k(x,q)=\frac{x^k}{1-x\witi M_k(x,q)}\prod_{j=1}^{k-1}\frac{\witi M_j(x,q)}{(1-x\witi M_j(x,q))^2},
\feq
where $\witi M_k(x,q)$ is defined recursively by the equation
\beq
\witi M_k(x,q)=\witi M_{k-1}(x,q)+\frac{xq(\witi M_{k-1}(x,q))^2}{1-x\witi M_{k-1}(x,q)}
\feq
with the initial condition $\witi M_1(x,q)=q.$
\end{proposition}
\begin{proof}[Proof of Proposition~\ref{mth1}]
In view of \eqref{eqP1} and \eqref{lm}, in order to prove the proposition it suffices to evaluate $L_k(x,q)$ and $M_k(x,q).$
These calculations are the content of the next two lemmas.
\begin{lemma}
\label{lem1}
For all $k\geq1,$
\beq
L_k(x,q)&=\prod_{j=1}^k\frac{1}{1-x\witi M_j(x,q)},
\feq
where $\witi M_k(x,q)$ satisfies the recurrence relation
\beq
\witi M_k(x,q)=\witi M_{k-1}(x,q)+\frac{xq(\witi M_{k-1}(x,q))^2}{1-x\witi M_{k-1}(x,q)},
\feq
with $\witi M_1(x,q)=q.$
\end{lemma}
\begin{proof}[Proof of Lemma~\ref{lem1}]
Any word $k\pi\in[k]^n$ can be written as
\beq
k\pi=k\pi^{(1)}k\pi^{(2)}\cdots k\pi^{(s)}
\feq
for some $s\geq1$ and subwords $\pi^{(j)}\in [k-1].$
Thus, the contribution for a fixed $s$ is $(x\witi M_k(x,q))^{s-1}\witi L_k(x,q)$, where
	\begin{align*}
	\witi L_k(x,q)&=\sum_{n\geq0}\sum_{\pi\in[k-1]^n}x^nq^{V(k\pi)},\\
	\witi M_k(x,q)&=\sum_{n\geq0}\sum_{\pi\in[k-1]^n}x^nq^{V(k\pi k)}.
	\end{align*}
	Hence,
	\begin{align}\label{eqL1}
	L_k(x,q)&=\sum_{s\geq1}(x\witi M_k(x,q))^{s-1}\witi L_k(x,q)=\frac{\witi L_k(x,q)}{1-x\witi M_k(x,q)}.
	\end{align}
	Note that any word $\pi\in[k-1]^n$ can be written as $\pi^{(0)}(k-1)\pi^{(1)}\cdots(k-1)\pi^{(s)}$ with $s\geq0$ and $\pi^{(j)}$ is a word over alphabet $[k-2]$ for all $j$. Thus,
	\begin{align}\label{eqLL1}
	\witi L_k(x,q)&=\sum_{s\geq1}(x\witi M_{k-1}(x,q))^{s-1}\witi L_{k-1}(x,q)=\frac{\witi L_{k-1}(x,q)}{1-x\witi M_{k-1}(x,q)},
	\end{align}
	where we used the fact that $V(k\pi'k)=V\big((k-1)\pi'(k-1)\big)$ for all $\pi'\in[k-2]^n$. Hence, by \eqref{eqL1} and \eqref{eqLL1}, we see that $\witi L_k(x,q)=L_{k-1}(x,q)$, which leads to
	\begin{align*}
	L_k(x,q)&=\frac{L_{k-1}(x,q)}{1-x\witi M_k(x,q)}.
	\end{align*}
	By induction on $k,$ and using the fact that $L_1(x,q)=\frac{1}{1-xq}$, we complete the proof for the formula $L_k(x,q)$.
	
	Now let us write an equation for $\witi M_k(x,q)$. Clearly, $\witi M_1(x,q)=q$, which counts the only empty word according the the visible pairs in $11$. Note that for any word $\pi\in[k-1]^n$, the word $k\pi k$ can be decomposed as $k\pi^{(0)}(k-1)\pi^{(1)}\cdots(k-1)\pi^{(s)}k$ with $\pi^{(j)}$ is a word over alphabet $[k-2]$ for all $j$. Thus,
	$$\witi M_k(x,q)=\witi M_{k-1}(x,q)+\sum_{s\geq1}x^sq(\witi M_k(x,q))^{s+1}
	=\witi M_{k-1}(x,q)+\frac{xq(\witi M_{k-1}(x,q))^2}{1-x\witi M_{k-1}(x,q)},$$
	where we used that fact $V(k\pi'(k-1))=V((k-1)\pi'(k-1))$ for all $\pi'\in[k-2]^n$.
\end{proof}

\begin{lemma}\label{lem2}
	For all $k\geq1$,
	$$M_k(x,q)=\frac{\witi M_k(x,q)}{1-x\witi M_k(x,q)}.$$
\end{lemma}
\begin{proof}[Proof of Lemma~\ref{lem2}]
	 For any word $k\pi\in[k]^n$, the word $k\pi(k+1)$ can be decomposed as either $k\pi'(k+1)$ or $k\pi'k\pi''(k+1),$ where $\pi'$ is a word over alphabet $[k-1]$ and $\pi''$ is a word over alphabet $[k]$. Since $V(k\pi'(k+1))=V(k\pi'k)$, we have
	$$M_k(x,q)=\witi M_k(x,q)+x\witi M_k(x,q)M_k(x,q),$$
	which, by solving for $M_k(x,q)$, complete the proof of Lemma~\ref{lem2}.
\end{proof}

	By Lemmas \ref{lem1} and \ref{lem2} and \eqref{eqP1}, we have
\beq
P_k(x,q)=x^k\prod_{j=1}^{k-1}\frac{1}{1-x\witi M_j(x,q)}\prod_{j=1}^{k-1}\frac{\witi M_j(x,q)}{1-x\witi M_j(x,q)}.
\feq
The proof of Proposition~\ref{mth1} is complete.
\end{proof}

\begin{example}
The first coefficients of the generating function $1+\sum_{k\geq1}P_k(x,q)$ are given by $1+x+2qx^2+5q^2x^3+(2q^4+13q^3)x^4+(18q^5+34q^4)x^5+(11q^7+103q^6+89q^5)x^6+(6q^9+160q^8+478q^7+233q^6)x^7
+(2q^{11}+206q^{10}+1359q^9+1963q^8+610q^7)x^8+(230q^{12}+3066q^{11}+8813q^{10}+7441q^9+1597q^8)x^9.$
\end{example}

With Proposition~\ref{mth1} at hand, we turn now to the study of the expected number of vertexes in $\calg_{n}$. More precisely, we obtain:

\begin{proposition}
\label{pro1}
For all $k\geq1$,
\begin{align*}
\frac{\partial}{\partial q}P_k(x,q)\,\Big|_{q=1}=\frac{x^k}{\prod_{j=1}^k(1-jx)}H_k(x),
\end{align*}
where
\beq
H_k(x)=\sum_{i=1}^{k-1}f_i(x)(1-ix)
	+2x\sum_{i=1}^{k-1}f_i(x)
	+xf_k(x),
\feq
with
\beq
f_i(x):=\frac{1+x\sum_{j=1}^{i-1}\frac{1-jx}{1-(j-1)x}}{(1-(i-1)x)(1-ix)}.
\feq
We use here the usual convention that an empty sum is zero.
\end{proposition}
\begin{proof}[Proof of Proposition~\ref{pro1}]
By Proposition \ref{mth1}, the generating function $\witi M_k(x,q)$ satisfies
\beq
\witi M_k(x,q)=\witi M_{k-1}(x,q)+\frac{xq(\witi M_{k-1}(x,q))^2}{1-x\witi M_{k-1}(x,q)}
\feq
with $\witi M_1(x,q)=q$. Thus,
$$\witi M_k(x,1)=\frac{\witi M_{k-1}(x,1)}{1-x\witi M_{k-1}(x,1)}$$
with $\witi M_1(x,1)=1$. Hence, by induction on $k$, we have
$\witi M_k(x,1)=\frac{1}{1-(k-1)x}$.

Moreover, by differentiation the recurrence relation at $q=1$, we obtain
\begin{align*}
\frac{\partial}{\partial q}\witi M_k(x,q)\mid_{q=1}&=
\frac{\partial}{\partial q}\witi M_{k-1}(x,q)\mid_{q=1}\\
&+\frac{x(\witi M_{k-1}(x,1))^2+x\witi M_{k-1}(x,1)\frac{\partial}{\partial q}\witi M_{k-1}(x,q)\mid_{q=1}(2-x\witi M_{k-1}(x,1))}{(1-x\witi M_{k-1}(x,1))^2},
\end{align*}
which, by $\witi M_k(x,1)=\frac{1}{1-(k-1)x}$, implies
$$\frac{\partial}{\partial q}\witi M_k(x,q)\mid_{q=1}=
\frac{x}{(1-kx)^2}+\frac{(1-(k-1)x)^2}{(1-kx)^2}\frac{\partial}{\partial q}\witi M_{k-1}(x,q)\mid_{q=1}.$$
We can now complete the proof of the proposition by using induction on $k$ and the initial condition $\frac{\partial}{\partial q}\witi M_1(x,q)\mid_{q=1}=1.$
\end{proof}

By Proposition \ref{pro1}, we have:
\begin{align*}
&\frac{\partial}{\partial q}P_k(x,q)\mid_{q=1}-\frac{x}{1-kx}\frac{\partial}{\partial q}P_{k-1}(x,q)\mid_{q=1}\\
&=\frac{x^k}{\prod_{j=1}^k(1-jx)}\left(
\frac{1+x\sum_{j=1}^{k-2}\frac{1-jx}{1-(j-1)x}}{1-(k-1)x}
+\frac{x+x^2\sum_{j=1}^{k-1}\frac{1-jx}{1-(j-1)x}}{(1-(k-1)x)(1-kx)}\right)
\end{align*}
with $\frac{\partial}{\partial q}P_1(x,q)\mid_{q=1}=\frac{x^2}{1-x}$. For all $k\geq2$, define
$$T_k(x)=\frac{x^k}{\prod_{j=1}^k(1-jx)}
\cdot\frac{1+x\sum_{j=1}^{k-2}\frac{1-jx}{1-(j-1)x}}{1-(k-1)x}.$$
Then,
\begin{align}\label{eqrecP}
(1-kx)\frac{\partial}{\partial q}P_k(x,q)\mid_{q=1}-x\frac{\partial}{\partial q}P_{k-1}(x,q)\mid_{q=1}&=\frac{T_k(x)+T_{k+1}(x)}{1-(k-1)x}
\end{align}
with $\frac{\partial}{\partial q}P_1(x,q)\mid_{q=1}=\frac{x^2}{(1-x)^2}$.

In order to solve \eqref{eqrecP}, we first study the corresponding exponential generating functions $\underline Q_k(x)$ and $\underline T_k(x)$ of the ordinary generating functions $Q_k(x) = \frac{\partial}{\partial q}P_k(x,q)\mid_{q=1}$ and $T_k(x)$, respectively. In other words,
\beq
 \underline Q_k(x)=\sum_{n\geq0}\frac{x^n}{n!}[x^n]\frac{\partial}{\partial q}P_k(x,q)\,\Big|_{q=1}
\feq
and
\beq
\underline T_k(x)=\sum_{n\geq0}\frac{x^n}{n!}[x^n]T_k(x), \qquad \qquad \underline T(x,y)=\sum_{k\geq2}\underline T_k(x)y^k.
\feq
\begin{lemma}\label{lem4}
The generating function $\underline T(x,y)=\sum_{k\geq2}\underline T_k(x)y^k$ is given by
\begin{align*}
\underline T(x,y)&=y^3\int_0^x(x-t)e^{ye^t-y}\int_0^tEi(1,ye^r)e^{ye^r+2r}drdt\\
&\qquad\qquad+y\int_0^x(t-x)e^{ye^t-y}(Ei(1,ye^t)e^{ye^t}(ye^t-1)-ye^t)dt\\
&\qquad\qquad\qquad\qquad+y(1-y)\int_0^x(t-x)e^{ye^t}Ei(1,ye^t)dt,
\end{align*}
where $Ei(1,z)=\int_1^{\infty}\frac{e^{-zt}}{t}dt$.
\end{lemma}
\begin{proof}[Proof of Lemma~\ref{lem4}]
By the definition of $T_k(x)$, we have:
$$(1-(k-3)x)(1-(k-1)x)T_k(x)-x(1-(k-3)xT_{k-1}(x)=\frac{x^{k+1}}{\prod_{j=1}^{k-3}(1-jx)}$$
with $L_2(x)=\frac{x^2}{1-x}$.
Rewriting this equation in terms of exponential generating functions, we obtain:
\begin{align*}
&\frac{d^4}{dx^4}\underline T_k(x)-(2k-4)\frac{d^3}{dx^3}\underline T_k(x)+(k-3)(k-1)\frac{d^2}{dx^2}\underline T_k(x)\\
&\qquad\qquad\qquad\qquad-\frac{d^3}{dx^3}\underline T_{k-1}(x)+(k-3)\frac{d^2}{dx^2}\underline T_{k-1}(x)=\frac{(e^x-1)^{k-3}}{(k-3)!},
\end{align*}
where we used \eqref{SnkIdent} and the fact that $\sum_{n\geq k}S_{n,k}\frac{x^n}{n!}=\frac{(e^x-1)^k}{k!}.$
\par
Multiplying both sides of the last recurrence by $y^k$ and summing over $k\geq 3$, we obtain:
\begin{align*}
&\frac{\partial^4}{\partial x^4}\big(\underline T(x,y)-\underline T_2(x)y^2\big)
-2y\frac{\partial^4}{\partial x^3\partial y}\big(\underline T(x,y)- \underline T_2(x)y^2\big)
+4\frac{\partial^3}{\partial x^3}\big(\underline T(x,y)- \underline T_2(x)y^2\big)
\\
&
\quad
+y\frac{\partial}{\partial y}\Big(y\frac{\partial^3}{\partial x^2\partial y}\big(\underline T(x,y)-\underline T_2(x)y^2\big)\Big)
-4y\frac{\partial^3}{\partial x^2\partial y}\big(\underline T(x,y)-\underline T_2(x)y^2\big)
\\
&
\quad
+3\frac{\partial^2}{\partial x^2}\big(\underline T(x,y)-\underline T_2(x)y^2\big)
-y\frac{\partial^3}{\partial x^3}\underline T(x,y)+y\frac{\partial^3}{\partial x^2\partial y}\big(y \underline T(x,y)\big)
\\
&
\quad
-3y\frac{\partial^2}{\partial x^2} \underline T(x,y)=y^3e^{y(e^x-1)},
\end{align*}
where $\underline T_2(x)=e^x-1-x$. Note that
\beq
\underline T(0,y)=\frac{\partial}{\partial x}\underline T(x,y)\mid_{x=0}=0, \qquad
\frac{\partial^2}{\partial x^2}\underline  T(x,y)\mid_{x=0}=y^2,\qquad
\frac{\partial^3}{\partial x^3}\underline T(x,y)\mid_{x=0}&=&y^2+y^3.
\feq
Solving the partial differential equation with these initial conditions, we obtain the result in Lemma~\ref{lem4}.
\end{proof}

Finally,

\begin{align*}
&\frac{d^2}{dx^2}\underline Q_k(x)-(2k-1)\frac{d}{dx}\underline Q_k(x)+k(k-1)\underline Q_k(x)\\
&-\frac{d}{dx}\underline Q_{k-1}(x)+(k-1)\underline Q_{k-1}(x)=\frac{d^2}{dx^2}(\underline T_k(x)+\underline T_{k+1}(x))
\end{align*}
with $\underline Q_1(x)=1+(x-1)e^x$.
\par
Recall $\underline Q(x,y)=\sum_{k\geq1}\underline Q_k(x)y^k$. Multiplying both sides of this recurrence equation by $y^k$ and summing over $k\geq2$, we obtain
\begin{align*}
&\frac{\partial^2}{\partial x^2}(\underline Q(x,y)-\underline Q_1(x)y)
-2y\frac{\partial^2}{\partial x\partial y}(\underline Q(x,y)-\underline Q_1(x)y) \\ &\quad +\frac{\partial}{\partial x}(\underline Q(x,y)-\underline Q_1(x)y) +y\frac{\partial}{\partial y}\left(y\frac{\partial}{\partial y}(\underline Q(x,y)-\underline Q_1(x)y) \right)\\
&\quad -y\frac{\partial}{\partial y}(\underline Q(x,y)-\underline Q_1(x)y)
-y\frac{\partial}{\partial x}\underline Q(x,y)
+y\frac{\partial}{\partial y}(y\underline Q(x,y)) -y\underline Q(x,y)\\
&=\frac{\partial^2}{\partial x^2}(\underline T(x,y)+1/y(\underline T(x,y)-\underline T_2(x)y^2))
\end{align*}
with $\underline Q(0,y)=0$ and $\frac{\partial}{\partial x}\underline Q(x,y)\mid_{x=0}=0$.
This along with Lemma \ref{lem4} and an aid of Maple, yields the explicit formula for the generating function $\underline Q(x,y)$ stated in Theorem~\ref{thm1}. \hfill\hfill\qed

\section{Proof of Theorem~\ref{thm2}}
\label{pr2}
The proof relies on the use of a generator of a uniformly random set partition of $[n]$ proposed by Stam \cite{stam1}.
We next describe Stam's algorithm for a given $n.$
\begin{enumerate}
\item For $m\in\nn,$ let $\mu_n(m)=\frac{m^n}{em! B_n}.$ Dobinski's formula \eqref{doob} shows that $\mu_n(\,\cdot\,)$ is a probability distribution on $\nn$.
\par
At time zero, choose a random $M\in\nn$ distributed according to $\mu_n,$ and arrange $M$ empty and unlabeled boxes.
\item Arranges $n$ balls labeled by integers from the set $[n].$
\par
At time $i\in [n],$ place the ball `$i$' into of one the $M$ boxes, chosen uniformly at random. Repeat until there are no balls remaining.
\item Label the boxes in the order that they get occupied by the balls. Once a box is labeled, the label does not change anymore.
\item Form a set partition $\pi$ of $[n]$ with $i$ in the $k$-th block if and only if ball `$i$" is in the $k$-th box.
\end{enumerate}
Let $N_i$ be the random number of nonempty boxes right after placing the $i$-th ball and $X_i$ be the label
of the box where the $i$-th ball was placed. Notice that if the $i$-th ball is dropped in an empty box, then
$X_i=N_{i-1}+1$ and $N_{i} = N_{i-1}+1.$ Otherwise, if the box was occupied previously, $X_i=X_j$ where $j<i$ is the first ball that was dropped in that box and $N_i = N_{i-1}$. Then, $X:=X_1\cdots X_n$ is the random set partition of $[n]$ produced  by the algorithm.
\par
We denote by $P_m(\,\cdot\,)$ conditional probability distribution $P(\,\cdot\,|\,M=m).$ Clearly $N_1=1$, $N_i\leq i,$ and
\beq
P_m(N_{i+1}=t+1|N_i=t) = \frac{m-t}{m}
\qquad
\mbox{\rm and}
\qquad
P_m(N_{i+1}=t|N_i=t) = \frac{t}{m}.
\feq
Let $\alpha_{i,t}(m):=P_m(N_i=t).$ Then, taking in account that
\beq
P_m(N_i=t) = P_m(N_i=t, N_{i-1}=t-1) + P_m(N_i=t, N_{i-1}=t),
\feq
we obtain:
\beq
\alpha_{i,t} (m)=
\left\{
\begin{array}{ll}
\frac{t}{m} \alpha_{i-1,t}(m) + \frac{m-t+1}{m}\alpha_{i-1,t-1}(m) & \text{if $2\leq t\leq m$ and $t\leq i$}
\\
[2mm]
0 & \text{if $t>i$ or $t>m$}
\\
[2mm]
\frac{1}{m^{i-1}} & \text{if $t=1$ and $1\leq i$}.
\end{array}
\right.
\feq
A comparison with \eqref{compa} reveals that for $t\leq m,$
\beqn \label{Nis}
P_m(N_i=t)=\frac{S_{i,t}}{m^i} \frac{m!}{(m-t)!}.
\feqn
In addition,
\beq
P_m(X_{i+1}=\ell|N_i=t) =
\left\{
\begin{array}{ll}
\frac{1}{m} & \quad \text{if} \quad \ell \leq t
\\
[2mm]
\frac{m-t}{m} & \quad \ell = t +1
\\
[2mm]
0 & \quad  \mbox{otherwise}.
\end{array}
\right.
\feq
Notice that some of the boxes may remain empty at the end of the algorithm's run.
\par
In view of \eqref{V_n_exp}, in order to calculate $E(V_n),$  we need to evaluate
\beq
E\big(e_n(i,j)\big) = E\big[E_M\big(e_n(i,j)\big)\big] =  E\Big(P_M\Big(\max_{i<\ell<j} X_\ell < \min\{X_i, X_j\}\Big)\Big)
\feq
for $(i,j) \in I_n^{(3)}.$ For any constant $m\in\nn$ we have:
\beqn \label{Tij-exact}
&&P_m\Big(\max_{i<\ell<j} X_\ell < \min\{X_i, X_j\}\Big)
\nonumber
\\
&& \quad = \sum_{t=1}^{(i-1) \wedge m}P_m\Big(\max_{i<\ell<j} X_\ell < \min\{X_i, X_j\}\,\Big|\,N_{i-1}=t\Big) P_m(N_{i-1}=t)
\nonumber
\\
&& \quad = \sum_{t=1}^{(i-1)\wedge m}\,\sum_{k=1}^{m\wedge (t+1)} P_m\Big(\max_{i<\ell<j} X_\ell < \min\{k, X_j\}\,\Big|\,N_{i-1}=t, X_i = k\Big)   \nonumber \\
&& \quad \qquad \qquad \qquad \qquad \times  P_m(X_i = k_i\,|\,N_{i-1}=t) P_m(N_{i-1}=t)  \nonumber \\
&& \quad = \frac{m!}{m^i} \sum_{t=1}^{(i-1)\wedge m}\frac{S(i-1,t)}{(m-t)!} \sum_{k=1}^t  P_m\Big(\max_{i<\ell<j} X_\ell < \min\{k, X_j\}\,\Big|\,N_{i-1}=t\Big) \nonumber \\
&& \qquad + \frac{m!}{m^i} \sum_{t=1}^{(i-1)\wedge (m-1)}  P_m\Big(\max_{i<\ell<j} X_\ell < \min\{t+1, X_j\}\,\Big|\,N_i=t+1\Big) \frac{S(i-1,t)}{(m-t-1)!}. \label{PTij}
\feqn
Furthermore, for any $a\leq t \leq m$ we have:
\beqn
&& P_m\Big(\max_{i<\ell<j} X_\ell < \min\{a, X_j\}\,\Big|\,N_i=t\Big)
\nonumber
\\
&&\quad = \sum_{b=1}^{a-1}P_m\Big(\max_{i<\ell<j} X_\ell < \min\{a, X_j\},X_{i+1}=b\,\Big|\,N_i=t\Big)
\nonumber
\\
&& \quad = \sum_{b=1}^{a-1} P_m\Big(\max_{i<\ell<j} X_\ell < \min\{a, X_j\}\,\Big|\,N_{i+1}=t\Big)  P_m(X_{i+1}=b\,|\,N_i=t)
\nonumber
\\
&& \quad =\frac{1}{m} \sum_{b=1}^{a-1} P_m\Big(\max_{i<\ell<j} X_\ell < \min\{a, X_j\}\,\Big|\,N_{i+1}=t\Big).
\nonumber
\feqn
Iterating, we obtain:
\beqn
&&
P_m\Big(\max_{i<\ell<j} X_\ell < \min\{a, X_j\}\,\Big|\,N_i=t\Big)
\nonumber
\\
&&
\qquad \qquad \qquad =\frac{1}{m^{j-i-1}} \sum_{b_{i+1}=1}^{a-1}\cdots \sum_{b_{j-1}=1}^{a-1} P_m\Big(\max_{i<\ell<j} b_\ell < X_j\,\Big|\,N_{j-1}=t\Big) \label{as_gen}.
\feqn
Denote $p:=\max_{i<\ell<j} b_\ell$ and $q:=|\{\ell\in (i,j):b_\ell=p \}|.$ In this terms, the last summation can be written as
\beq
&&\sum_{b_{i+1}=1}^{a-1}\cdots \sum_{b_{j-1}=1}^{a-1} P_m\Big(\max_{i<\ell<j} b_\ell < X_j\,\Big|\,N_{j-1}=t\Big)
\\
&&
\qquad
= \sum_{p=1}^{a-1} \sum_{q=1}^{j-i-1} \binom{j-i-1}{q} (p-1)^{j-i-1-q}P_m(X_j>p|N_{j-1}=t) \\
&&
\qquad
= \sum_{p=1}^{a-1} \left( p^{j-i-1} - (p-1)^{j-i-1}\right) \left(1 -\frac{p}{m} \right)
= \Big(1-\frac{a}{m}\Big)(a-1)^{j-i-1}+\frac{1}{m}\Psi_{i-j}(a-1),
\feq
where $\Psi_{i-j}$ is introduced in \eqref{rho}. Thus,
\beqn
P_m\Big(\max_{i<\ell<j} X_\ell < \min\{a, X_j\}\,\Big|\,N_i=t\Big)   =
\frac{1}{m^{j-i}} \left( (m-a)(a-1)^{j-i-1}+\Psi_{i-j}(a-1) \right). \label{as_gen2}
\feqn
\par
Inserting \eqref{Nis} and \eqref{as_gen2} into \eqref{PTij} and taking expectation with respect to $\mu_n(\,\cdot\,)$, we obtain:
\beqn
&&eB_n P(e_n(i,j)=1) = \sum_{t=1}^{i-1} S_{i-1,t}\Psi_{i-j}(t-1)  \sum_{m=t}^{\infty} \frac{m^{n-j+1}}{(m-t)!}   \nonumber   \\
&& \qquad\qquad + \sum_{t=1}^{i-1} S_{i-1,t}  \sum_{a=1}^t \left( -a(a-1)^{j-i-1}+\Psi_{i-j}(a-1) \right) \sum_{m=t}^{\infty}\frac{m^{n-j}}{(m-t)!} \nonumber \\
&&\qquad\qquad  + \sum_{t=1}^{i-1}S_{i-1,t}t^{j-i-1} \sum_{m=t+1}^\infty  \frac{m^{n-j+1}}{(m-t-1)!} \nonumber \\
&&\qquad\qquad  + \sum_{t=1}^{i-1} S_{i-1,t}  \left(-(t+1)t^{j-i-1}+\Psi_{i-j}(t)\right)\sum_{m=t+1}^\infty \frac{m^{n-j}}{(m-t-1)!}, \nonumber
\feqn
as desired.
\hfill\hfill\qed
\section{Proof of Theorem~\ref{thm3}}
\label{thm3p}
Write:
\beq
P\big((i,j)\in \calv^w_n\setminus \calv_n\big) = E\Big(P_M\Big(\max_{i<\ell<j}X_\ell = \min\{X_i,X_j\}\Big)\Big).
\feq
Case I) If $(i,j)\in I_n^{(3)}$, then similarly to the calculation in \eqref{Tij-exact}, for any $m\in\nn$ we have:
\beqn
&&P_m\Big(\max_{i<\ell<j}X_\ell = \min\{X_i,X_j\}\Big)
\nonumber
\\
&&
\quad
= \frac{m!}{m^i} \sum_{t=1}^{(i-1)\wedge m}\frac{S_{i-1,t}}{(m-t)!}
\sum_{a=1}^t
P_m\Big(\max_{i<\ell<j} X_\ell = \min\{a, X_j\}\,\Big|\,N_i=t\Big)  \nonumber \\
&&
\quad
\quad + \frac{m!}{m^i} \sum_{t=1}^{(i-1)\wedge (m-1)} P_m\Big(\max_{i<\ell<j} X_\ell = \min\{t+1, X_j\}\,\Big|\,N_i=t+1\Big)
\frac{S_{i-1,t}}{(m-t-1)!}.
\label{PTwij}
\feqn
Similarly to \eqref{as_gen}, for $a\leq t \leq m$ we have:
\beqn
&& P_m\Big(\max_{i<\ell<j} X_\ell = \min\{a, X_j\}\,\Big|\,N_i=t\Big)
=
\nonumber
\\
&&
\quad
 P_m\Big(\max_{i<\ell<j} X_\ell =a,\, X_j\geq a\,\Big|\,N_i=t\Big) +
P_m\Big(\max_{i<\ell<j} X_\ell =X_j, \, X_j<a\,\Big|\,N_i=t\Big). \label{pmxl}
\feqn
The first term on the right hand-side of \eqref{pmxl} can be written as
\beqn
&& \quad P_m\Big(\max_{i<\ell<j}X_\ell=a\,\Big|\,N_i = t\Big)P_n\big(a \leq  X_j\big|N_{j-1}=t\big)
\nonumber
\\
&&
\qquad\qquad
= \Big\{P_m\Big(\max_{i<\ell<j}X_\ell\leq a\,\Big|\,N_i = t\Big)-P_m\Big(\max_{i<\ell<j}X_\ell\leq a-1\,\Big|\,N_i = t\Big)\Big\}\frac{m-a+1}{m}
\nonumber
\\
&&
\qquad\qquad = \frac{1}{m^{j-i}}(m-a+1)\big(a^{j-i-1}-(a-1)^{j-i-1}\big). \label{PmXla}
\feqn
Similarly, the second term in right hand side of \eqref{pmxl} contributes:
\beqn
&&\sum_{b=1}^{a-1} P_m\Big(\max_{i<\ell<j}X_\ell=b,\,X_j=b\,\Big|\,N_i = t\Big)
\nonumber
\\
&&
\qquad\qquad
= \sum_{b=1}^{a-1} P_m\Big(\max_{i<\ell<j}X_\ell=b\,\Big|\,N_i = t\Big)P_m(X_j = b|N_{j-1}=t)
\nonumber
\\
&&
\qquad\qquad
= \frac{1}{m}\sum_{b=1}^{a-1} P_m\Big(\max_{i<\ell<j}X_\ell=b\,\Big|\,N_i = t\Big)
\nonumber
\\
&&
\qquad\qquad
= \frac{1}{m} \sum_{b=1}^{a-1} \Big\{\Big(\frac{b}{m}\Big)^{j-i-1}-\Big(\frac{b-1}{m}\Big)^{j-i-1}\Big\}
= \frac{(a-1)^{j-i-1}}{m^{j-i}}. \label{PmXla2}
\feqn
Inserting \eqref{PmXla} and \eqref{PmXla2} back into \eqref{pmxl}, we obtain:
\beq
&&
P_m\Big(\max_{i<\ell<j}X_\ell=\min\{a, X_j\}\,\Big|\,N_i = t\Big)
\\
&&
\qquad
\qquad
= \frac{1}{m^{j-i}}\left( (m-a+1)a^{j-i-1} - (m-a)(a-1)^{j-i-1} \right).
\feq
Plugging the result into \eqref{PTwij} and taking expectation with respect to $\mu_n(\,\cdot\,)$ gives:
\beq
&&
eB_n P\big((i,j)\in \calv^w_n\setminus \calv_n\big)
\\
&&
\qquad
=
\sum_{m=1}^{\infty} \sum_{t=1}^{(i-1)\wedge m} \frac{m^{n-j}}{(m-t)!} S_{i-1,t}
\Big(mt^{j-i-1}-t^{j-i} + t^{j-i-1}+2\sum_{a=1}^{t-1} a^{j-i-1}\Big)
\\
&&
\qquad\quad
+ \sum_{m=1}^\infty \sum_{t=1}^{(i-1)\wedge (m-1)} \frac{m^{n-j}}{(m-t-1)!} S_{i-1,t} \left( (m-t)\big((t+1)^{j-i-1}-t^{j-i-1}\big) + t^{j-i-1} \right).
\feq
The result in case (i) follows from this formula by changing the order of summation and applying \eqref{Blextension}.
\\\
$\mbox{}$
\\
Case (ii) If $(i,j) \in I^{(1)}_n,$ then
\beq
&&P_m\Big(\max_{i<\ell<j}X_\ell = \min\{1,X_j\}\Big) = \frac{1}{m^{j-i-1}}.
\feq
Hence, an application of Dobinski's identity \eqref{doob} yields
\beq
P\big((i,j)\in \calv^w_n\setminus \calv_n\big) = \frac{1}{eB_n} \sum_{m=1}^\infty  \frac{m^{n-j+i+1}}{m!} = \frac{B_{n-j+i+1}}{B_n},
\feq
as desired.
\hfill\hfill\qed

\end{document}